\documentclass[10pt]{amsart}

\usepackage{lmodern}
\usepackage[T1]{fontenc}

\usepackage{amsmath, amsthm, amscd, amssymb, xcolor}
\definecolor{dblue}{rgb}{0,0,.6}
\usepackage[colorlinks=true, linkcolor=dblue, citecolor=dblue, filecolor = dblue, menucolor = dblue, urlcolor = dblue]{hyperref}
\usepackage{tikz-cd} 

\usepackage[all,cmtip]{xy}

\newcommand{\Q}{\mathbb{Q}}

\newcommand{\C}{\mathbb{C}}
\newcommand{\N}{\mathbb{N}}
\newcommand{\sing}{\operatorname{sing}}
\newcommand{\supp}{\operatorname{supp}}

\newcommand{\sm}{\operatorname{sm}}

\newcommand{\Ne}{\operatorname{N\'e}}

\newcommand{\bbQ}{\mathbb{Q}}
\newcommand{\bbR}{\mathbb{R}}
\newcommand{\bbC}{\mathbb{C}}
\newcommand{\bbH}{\mathbb{H}}

\newcommand{\bbP}{\mathbb{P}}
\renewcommand{\AA}{\mathcal{A}}

\newcommand{\DD}{\mathcal{D}}
\newcommand{\FF}{\mathcal{F}}
\newcommand{\HH}{\mathcal{H}}

\newcommand{\OO}{\mathcal{O}}

\newcommand{\VV}{\mathcal{V}}

\newcommand{\XX}{\mathcal{X}}

\newcommand{\frso}{\mathfrak{so}}
\newcommand{\GL}{\mathrm{GL}}

\newcommand{\SO}{\mathrm{SO}}
\newcommand{\Sp}{\mathrm{Sp}}

\renewcommand{\ge}{\geqslant}
\renewcommand{\le}{\leqslant}

\newcommand{\catK}{\mathsf{(PVHS_{K3})}}
\newcommand{\catAb}{\mathsf{(VHS_{Ab})}}
\newcommand{\catlimK}{\mathsf{(PMHS_{K3})}}
\newcommand{\catllimK}{\mathsf{(PMHS_{K3}^{lim})}}
\newcommand{\catlimAb}{\mathsf{(MHS_{Ab})}}
\newcommand{\KS}{\mathrm{KS}}
\newcommand{\limKS}{\mathrm{KS^{lim}}}
\newcommand{\LimK}{\mathrm{Lim_{K3}}}
\newcommand{\LimAb}{\mathrm{Lim_{Ab}}}

\newcommand{\st}{\enskip |\enskip}
\newcommand{\sdot}{{\raisebox{0.16ex}{$\scriptscriptstyle\bullet$}}}

\newcommand{\emrp}{\mathrm{End}}

\newcommand{\im}{\mathrm{im}}

\newcommand{\ii}{\sqrt{-1}}

\newcommand{\wdg}{\wedge}

\newcommand{\Cl}{\mathcal{C}l}

\newcommand{\Spin}{\mathrm{Spin}}

\newcommand{\hrarr}{\hookrightarrow}

\newcommand{\Spec}{\mathrm{Spec}}

\newcommand{\End}{\mathrm{End}}

\newtheorem{defn}{Definition}[section]
\newtheorem{prop}[defn]{Proposition}
\newtheorem{thm}[defn]{Theorem}
\newtheorem{lem}[defn]{Lemma}
\newtheorem{cor}[defn]{Corollary}

{\theoremstyle{remark}
  \newtheorem{rem}[defn]{Remark}
  
}

\makeatletter
\@addtoreset{equation}{section}
\makeatother

\title{The Kuga--Satake construction under degeneration}

\author{Stefan Schreieder}
\address{Mathematisches Institut, LMU M\"unchen,  Theresienstr.\ 39, 80333 M\"unchen, Germany}
\email{schreieder@math.lmu.de}
\author{Andrey Soldatenkov}
\address{Institut f\"ur Mathematik, Humboldt-Universit\"at zu Berlin, Unter den Linden 6, 10099 Berlin, Germany.}
\email{soldatea@hu-berlin.de}
\date{March 27, 2019}

\thanks{The authors are supported by the SFB/TR 45 `Periods, Moduli Spaces and Arithmetic of Algebraic Varieties'
of the DFG (German Research Foundation).}

\begin{document}

\begin{abstract}  
We extend the Kuga--Satake construction to the case of limit mixed Hodge structures of K3 type.
We use this to study the geometry and Hodge theory of degenerations of Kuga--Satake abelian varieties associated to polarized variations of K3 type Hodge structures over the punctured disc.
\end{abstract}

\maketitle

\section{Introduction} 

The Kuga--Satake construction \cite{KS} associates to any polarized rational weight two Hodge structure $V$ of K3 type (i.e.\ with $V^{2,0}\cong \C$) an abelian variety $A$, well-defined up to isogeny, with an embedding of Hodge structures 
$$
\mathrm{ks}: V(1)\hookrightarrow \End (H^1(A,\Q))\subset H^2(A\times A,\Q)(1).
$$ 
Here the rational vector space $H^1(A,\Q)$ is given by the Clifford algebra $\Cl(V,q)$, associated to the polarization $q$ of $V$.
If $V\subset H^2(X,\Q)$ for some smooth projective variety $X$ (e.g.\ a K3 surface or, more generally, a projective hyperk\"ahler manifold), then the above embedding corresponds to a Hodge class on  $X\times A\times A$. 
Even though algebraicity of that class is known only in very few cases, the Kuga--Satake construction is expected to give a close relation between the geometry of $X$ and the associated Kuga--Satake abelian variety $A$.
The  Kuga--Satake construction has been generalized by Voisin \cite{V} and  
by Kurnosov, Verbitsky and the second author \cite{KSV}.

The Kuga-Satake abelian varieties may be seen as analogues of intermediate Jacobians for K3 type Hodge
structures. If we apply this construction to families of K3 surfaces, it is important to
understand its behaviour near the points where the surfaces become singular. For intermediate Jacobians
this is a classical and well-studied subject, see e.g. \cite{Cl}, \cite{Sa}, \cite{Zu}.
This motivates the study of the Kuga--Satake construction under degeneration.
We start from a polarized variation of Hodge structures (VHS) of K3 type over the punctured
disc $\Delta^* = \Delta\setminus\{0\}$; up to a finite base change, we may assume that the monodromy of the underlying local system is unipotent, see \cite[Lemma 4.5]{Sch}.
Geometrically such a VHS comes from a flat projective family over the unit disc
$\pi\colon \mathcal X\to \Delta$, smooth over $\Delta^*$ and with general fibre 
 for instance a projective hyperk\"ahler manifold
or an abelian surface. 
Applying the Kuga--Satake construction to the VHS over $\Delta^*$ we get a polarized variation of weight one
Hodge structures. 
Using a result of Borel \cite{Bor} and the semi-stable reduction theorem \cite{KKMSD}, we obtain (up to a finite base change)
a semi-stable family $\alpha\colon\mathcal A\to \Delta$ of abelian varieties.
If the fibre $\XX_0$ is singular, then we expect $\AA_0$ to be singular as well, and it is natural to wonder
how to describe such singular fibres. When $\XX$ is a family of K3 surfaces, this question is well
understood since the work of Kulikov \cite{Ku}: there are essentially three types of fibres $\XX_0$ that can appear.
The analogous question for the associated Kuga--Satake varieties 
is however much more subtle. 

As a first approximation, one can try to describe the mixed Hodge structure of $\AA_0$.
Via the Clemens--Schmid exact sequence, this is essentially governed by the limit mixed Hodge
structure on a smooth fibre $\AA_{t}$, where $t\in \Delta^*$ is some base point.
The limit mixed Hodge structure $H^1_{lim}(\AA_t,\Q)$ has as underlying vector space $H^1(\AA_t,\Q)$
and is given by two filtrations: the Hodge filtration $F^\sdot_{lim}$ on $H^1(\AA_t,\C)$ and the
weight filtration $W_\sdot$ on $H^1(\AA_t,\Q)$.
The weight filtration is induced by the monodromy operator given by parallel transport along
a loop in $\Delta^*$. The limit Hodge filtration is determined by the germ of the variation of Hodge
structures on $\AA_s$ near zero. This filtration is not canonical and depends on the choice
of the local coordinate on $\Delta^*$; in what follows we fix the local coordinate and do
not consider this dependence.


In this paper we show that the limit mixed Hodge structure $H^1_{lim}(\AA_t,\Q)$ depends only
on the limit mixed Hodge structure attached to the initial VHS of K3 type.
Roughly speaking, this says that the limit of the Hodge structures on the Kuga--Satake
side does not depend on the individual Hodge structures $H^1(\AA_s,\Q)$ for $s\in \Delta^\ast$,
but only on the limit of the Hodge structures on the K3 side.
In fact, we show more generally that the Kuga--Satake construction extends from the case of
pure Hodge structures of K3 type to the case of limit mixed Hodge structures of K3 type
(cf.\ Section \ref{subsubsec:MHS-K3} below) and this construction is compatible
with the geometric situation described above.

All Hodge structures that we consider in this paper are rational and have level $\le 2$.
We consider the following categories, for precise definitions see Section \ref{sec:preliminaries} below: 
\begin{itemize}
\item $\catK$ = category of polarized VHS of K3 type and with unipotent monodromy over $\Delta^*$;
\item $\catAb$ = category of polarizable VHS of weight one and with unipotent monodromy over $\Delta^*$;
\item  $\catlimK$ = category of polarized MHS of K3 type;
\item  $\catlimAb$ = category of polarizable MHS of weight one.
\end{itemize}
The polarized MHS here are in the sense of \cite[Definition 2.26]{CKS},
see also Section \ref{subsubsec:MHS-K3} below.
For us it will be important that polarizations are fixed for only half of the above categories, see Remark \ref{rem:polarizations} below.
The above categories are related by the following diagram of functors:
\begin{align} \label{diag:functors}
\xymatrix{
\catK \ar[d]^{\LimK} \ar[r]^{\KS} &\catAb\ar[d]^{\LimAb}\\
\catlimK  
&\catlimAb}
\end{align}
where $\KS$ denotes the Kuga--Satake functor described above, and $\LimK$ and $\LimAb$ denote the functors that compute the corresponding limit mixed Hodge structures. We will denote by $\catllimK$ the essential image of $\LimK$, i.e.\ the full subcategory 
of $\catlimK$ whose objects are in the image of $\LimK$.

\begin{thm}\label{thm:limit-KS}
There exists a functor $\limKS\colon\catllimK\to\catlimAb$ which makes the diagram (\ref{diag:functors}) commutative; that is,
$$\LimAb\circ\KS = \limKS\circ \LimK.$$ 
Moreover, for any polarized limit mixed Hodge structure $\overline{V} = (V,q,F^\sdot,N)\in \catllimK$ of K3 type, 
there exists an embedding of mixed Hodge structures 
$$
\mathrm{ks}: \overline{V}(1)\hookrightarrow \End(\overline{H}), \ \ \text{where}\ \overline{H}=\limKS(\overline{V}).
$$ 
\end{thm}



Let us emphasize that the functor $\limKS$ is defined only on the essential image of $\LimK$.
We do not claim that it extends in any natural way to the whole category $\catlimK$.

For any polarized limit mixed Hodge structure $\overline{V} = (V,q,F^\sdot,N)\in \catllimK$ of K3 type, the limit mixed Hodge structure $\limKS(\overline{V})=(H,F_{KS}^\sdot,N_{KS})$ of abelian type in the above theorem has as underlying $\Q$-vector space the Clifford algebra $H:=\Cl(V,q)$.
The Hodge filtration on $H$ is determined by the half-dimensional subspace
$$
F^1_{KS}H_\C:=F^2V_\C\cdot H_\C ,
$$
where $F^2V_\C\cdot H$ denotes the right ideal in the Clifford algebra $H_\C$, generated by the one-dimensional subspace $F^2V_\C$, cf.\ Section \ref{sec:Hodgefiltration} below.
This description of the Hodge filtration relies on a simple description of the usual Kuga--Satake construction, which might be of independent interest, see Lemma \ref{lem:KS-alternative} below. 
Finally, the nilpotent operator $N_{KS}$ is zero if $N=0$ and it is given by left multiplication with the element $f_1f_2\in H$, where $f_1,f_2\in V$ form a basis of $\im(N:V\to V)$, which turns out to be two-dimensional whenever $N\neq 0$, cf.\ Proposition \ref{prop:weight-filtration} below.
As usual, the weight filtration $W_\sdot$ on $H$ is then given by $W_0H=\im (N_{KS})$, $W_1H=\ker(N_{KS})$ and $W_2H=H$, see e.g.\ \cite{Mo}.

Together with the Clemens--Schmid sequence, the above result allows us to classify completely
the first cohomology groups of degenerations of Kuga--Satake varieties. Consider a polarized VHS
$\overline{\VV} = (\VV,q,\FF^\sdot)\in \catK$ and let $T = e^N$ be the monodromy transformation of
the local system $\VV$. It is known that $N^3 = 0$. We will say that $\overline{\VV}$ is of type I
if $N=0$, of type II if $N\neq 0$, $N^2=0$ and of type III if $N^2\neq 0$. The rank of $\overline{\VV}$
is by definition the rank of the local system $\VV$.

\begin{thm} \label{thm:type} 
Let $\overline{\VV} = (\VV,q,\FF^\sdot)\in \catK$ be a polarized VHS of K3 type of rank $r$,
with associated semi-stable family of Kuga--Satake varieties $\alpha:\mathcal A\to \Delta$,
smooth over the punctured disc $\Delta^\ast$ and with central fibre $\AA_0$. 
Then one of the following holds.
\begin{enumerate}
\item If $\overline{\VV}$ is of type I, then $\alpha$ is birational
to a smooth projective family of abelian varieties over $\Delta$ and the
mixed Hodge structure on $H^1(\AA_0,\Q)$ is pure with Hodge numbers
$$
{\arraycolsep=0.2em
\begin{array}{ccc}
 & 0 & \\
 2^{r-1}& & 2^{r-1} \\
 &0 &
\end{array}}
$$

\item If $\overline{\VV}$ is of type II, then the mixed Hodge structure on $H^1(\AA_0,\Q)$ has Hodge numbers
$$
{\arraycolsep=0.1em
\begin{array}{ccc}
 & 0 & \\
 2^{r-2}& & 2^{r-2} \\
 & 2^{r-2} &
\end{array}}
$$

\item If $\overline{\VV}$ is of type III, then the mixed Hodge structure on $H^1(\AA_0,\Q)$ is of weight zero with Hodge numbers
$$
{\arraycolsep=0.5em
\begin{array}{ccc}
 & 0 & \\
0& & 0 \\
 &2^{r-1} &
\end{array}}
$$
\end{enumerate} 
\end{thm}

We remark that the semi-stable family $\alpha$ in the above theorem exists after base change and its restriction to $\Delta^\ast$ is unique up to isogeny, see Section \ref{subsec:catAb} below.

A natural invariant associated to any semi-stable family $\alpha:\mathcal A\to \Delta$ of Kuga--Satake abelian varieties is the dual complex $\Sigma$ of the central fibre $\AA_0$.
By \cite{ABW}, the homotopy type of $\Sigma$ depends only on the restriction of $\mathcal A$ to the punctured disc $\Delta^*$ and so it does not depend on the chosen semi-stable model. 

As a consequence of our results, we are able to compute the rational cohomology algebra of $\Sigma$ explicitly.

\begin{cor} \label{cor:dualcomplex} 
In the notation of Theorem \ref{thm:type}, let  $\Sigma$ be the dual complex of the central fibre $\AA_0$.
\begin{enumerate}
\item If $\overline{\VV}$ is of type I, then $\Sigma$ is homotopy equivalent to a point.
\item If $\overline{\VV}$ is of type II, then $H^\ast(\Sigma,\bbQ)\simeq H^\ast(T,\bbQ)$, where $T=(S^1)^{w}$ is a real torus of real dimension $w=2^{r-2}$.
In particular,  the central fibre $\AA_0$ has at least $2^{r-2}$ components.
\item If $\overline{\VV}$ is of type III, then $H^\ast(\Sigma,\bbQ)\simeq H^\ast(T,\bbQ)$, where $T=(S^1)^{w}$ is a real torus of real dimension $w=2^{r-1}$.
In particular,  the central fibre $\AA_0$ has at least $2^{r-1}$ components.
\end{enumerate} 
\end{cor}

As a consequence of our construction, the semi-stable family $\alpha\colon \mathcal A\rightarrow \Delta$ of Kuga--Satake abelian varieties from Theorem \ref{thm:type} will automatically be projective over the disc, that is 
 %
$\AA\subset \bbP^n\times \Delta$
for some $n$, see the discussion in Section \ref{subsec:catAb}.
Under this assumption, it is possible to replace the analytic family $\alpha$ by an algebraic one over the formal disc $\Spec R$, where $R=\C[[t]]$ denotes the ring of formal power series. 
Indeed, $\mathcal A\subset \bbP^n\times \Delta$ is cut out by finitely many polynomials whose coefficients are complex analytic functions on the disc and so we may regard them as elements of $R$, cf.\ Section \ref{subsec:catAb} below.
This defines a projective scheme $\AA_R\subset \bbP^n_R$, flat over the formal disc $\Spec R$. 
The base change of $\AA_R$ to the fraction field $K=\C((t))$ will be denoted by $\mathcal A_K$; it is an abelian variety over $K$.
Note that the special fibre $\mathcal A_R\times_R\C$ of $\mathcal A_R\to \Spec R$ coincides with the central fibre $\mathcal A_0$ of $\mathcal A\to \Delta$.

Blowing-up a smooth subvariety in the smooth locus of $\mathcal A_0$ shows that the semi-stable model $\mathcal A\to \Delta$ is not unique, in fact, the isomorphism type of each component of the special fibre as well as the number of such components depends on the particular choice of a semi-stable model.
Instead of semi-stable models, it is thus more convenient to work with the N\'eron model, which is canonical.
The N\'eron model $\mathcal A^{\Ne}\to \Spec R$ is a quasi-projective commutative group scheme over $R$
with generic fibre $\mathcal A^{\Ne}_K$ isomorphic to $\mathcal A_K$ and such that for any smooth separated $R$-scheme $X$ with generic fibre $X_K:=X\times_RK$, any morphism $X_K \to \mathcal A^{\Ne}_K$ extends to a unique $R$-morphism $X\to \mathcal A^{\Ne}$. 
For abelian varieties over $K$, N\'eron models exist and are unique up to unique isomorphism, cf.\ \cite{BLR}.

As pointed out to us by Johannes Nicaise, Halle and Nicaise \cite{HN2} showed that 
the special fibre $\mathcal A^{\Ne}_0:=\mathcal A^{\Ne}\times_R\C$
of the N\'eron model can be described in terms of the limit mixed Hodge structure of the family $\mathcal A\to \Delta$. 
Moreover, according to \cite{HN1} and \cite{HN2}, the limit mixed
Hodge structure essentially determines the motivic zeta function $Z_{\AA_K}\in (K_0(\mathcal{V}ar_\bbC)[\mathbb{L}^{-1}])[[T]]$ of the
abelian variety $\AA_K$. 
This zeta function is a formal power series with coefficients in the
Grothendieck ring of varieties localized at the class of the affine line; 
for a precise definition
of $Z_{\AA_K}$, see \cite[Section 2]{HN1}. 
For K3 surfaces over $K$ with semi-stable model over $R$, this zeta function has already been computed by Stewart and Vologodsky \cite{SV}. 


\begin{cor}\label{cor:Neron}
Let $\alpha\colon \mathcal A\rightarrow \Delta$ be the family of Kuga-Satake abelian
varieties as in the Theorem \ref{thm:type}. 
Then the special fibre $\mathcal A^{\Ne}_0$
of the N\'eron model 
is a disjoint union
of isomorphic components $\mathcal A^{\Ne}_0=\bigsqcup_i A$, where $A$ is a semi-abelian variety
given by an extension
$$
0\longrightarrow (\C^\ast)^{w}\longrightarrow A \longrightarrow B\longrightarrow 0,
$$
where $B$ is an abelian variety with rational weight one Hodge structure isomorphic to
$gr^W_1(\limKS(\overline {\mathcal V}))$ and $w$ is the dimension of $gr^W_0(\limKS(\overline {\mathcal V}))$. In particular:
\begin{enumerate}
\item If $\overline {\mathcal V}$ is of type I, then $w = 0$ and $A$ is an abelian variety of dimension $2^{r-1}$;
\item If $\overline {\mathcal V}$ is of type II, then $w = 2^{r-2}$ and $B$ is of dimension $2^{r-2}$;
\item If $\overline {\mathcal V}$ is of type III, then $w = 2^{r-1}$, $B$ is trivial and $A$ is an algebraic torus.
\end{enumerate}
The motivic zeta-function (see \cite{HN1}) of the abelian variety $\AA_K$ is given by
$$
Z_{\AA_K}(T) = N[B](\mathbb{L}-1)^w \sum_{d\ge 1}d^w T^d,
$$
where $N$ is the number of connected components of the special fibre $\AA^{\Ne}_0$
and $[B]$ denotes the class of $B$ in $K_0(\mathcal{V}ar_\bbC)$.
\end{cor}

By \cite[Theorem 1.4]{JM} (see also \cite{BLR}), there is a close relationship between the semi-stable model $\mathcal A$ and the N\'eron model $\mathcal A^{\Ne}$, which we describe next.
To this end, 
note that the canonical bundle
$K_{\mathcal A}$ is trivial away from the central fibre $\mathcal A_0$ and so we can write 
$$
K_{\mathcal A}\sim \sum_{i}a_i\mathcal A_{0i},
$$
where $\mathcal A_{0i}$ denote the components of $\mathcal A_0$ and we may assume that $a_i\geq 0$ for all $i$
and $a_i=0$ for at least one $i$.
We then define the support of $K_{\mathcal A}$ as 
$$
\supp(K_{\mathcal A}):=\bigcup_{i\colon a_i\neq 0}\mathcal A_{0i} .
$$
Note that $\supp(K_{\mathcal A})$ is empty if $K_{\mathcal A}$ is trivial; such models are called good models in \cite{JM}.

We further consider 
$$
\mathcal A^{\sm}:=\AA_R\setminus \AA_0^{\sing} 
$$ 
and
$$
\mathcal A^{\mathrm{mo}}:=\AA^{\sm}\setminus \supp(K_{\AA}).
$$

By \cite[Theorem 1.4]{JM} we have an open immersion $\AA^{\mathrm{mo}}\hrarr \AA^{\Ne}$ that gives
a one-to-one correspondence between components of the special fibres. When $\AA$ is a good model,
we have $\AA^{\Ne}\simeq \AA^{\sm}\simeq \AA^{\mathrm{mo}}$.  
Using this description, Corollary \ref{cor:Neron} implies the following.

\begin{cor}\label{cor:supp}
In the notation of Theorem \ref{thm:type}, 
any component $\mathcal A_{0i}$ of the central fibre $\mathcal A_0$ that is not contained in the support of $K_{\mathcal A}$ is up to birational equivalence given as follows.
\begin{enumerate}
\item If $\overline {\mathcal V}$ is of type I, then $\mathcal A_{0i}$ is birational to an abelian variety of dimension $2^{r-1}$, which  up to isogeny is uniquely determined by the pure weight one Hodge structure $gr^W_1(\limKS(\overline {\mathcal V}))$.
\item If $\overline {\mathcal V}$ is of type II, then $\mathcal A_{0i}$ is birational to a $\mathbb P^{2^{r-2}}$-bundle over an abelian variety of dimension $2^{r-2}$, which up to isogeny is uniquely determined by the pure weight one Hodge structure $gr^W_1(\limKS(\overline {\mathcal V}))$.
\item If $\overline {\mathcal V}$ is of type III, then $\mathcal A_{0i}$ is rational.
\end{enumerate}
\end{cor}

\section{Preliminaries} \label{sec:preliminaries}

\subsection{The four categories}

We denote by $\Delta$ the unit disc, $\Delta^* = \Delta\setminus\{0\}$ the punctured unit disc  
and $\tau\colon \bbH\to \Delta^*$ denotes the universal covering, where $\bbH$ is the upper half-plane.
We further fix a base point $t\in \Delta^\ast$ and consider the following categories.

\subsubsection{} The category $\catK$ of polarized variations of Hodge structures of K3 type over $\Delta^*$:
its objects are triples $(\VV,q,\FF^\sdot)$, where $\VV$ is a local system of $\bbQ$-vector spaces
over $\Delta^*$, $q\in \Gamma(\Delta^*,S^2\VV^*)$ is a polarization, and $\FF^\sdot$ is a decreasing
filtration of $\VV\otimes\OO_{\Delta^*}$ by holomorphic subbundles. These structures must satisfy the following
conditions: the monodromy transformation of $\VV$ is unipotent, $q$ has signature $(2,r-2)$,
the filtration is of the form $0 = \FF^3\VV\subset \FF^2\VV\subset\FF^1\VV\subset \FF^0\VV = \VV\otimes\OO_{\Delta^*}$,
where $\FF^2\VV$ is of rank one, $\FF^1\VV = (\FF^2\VV)^\perp$, and for any non-vanishing local section $\sigma$
of $\FF^2\VV$, we have $q(\sigma,\sigma) = 0$ and $q(\sigma,\bar{\sigma})>0$. The morphisms in $\catK$ are
morphisms of local systems that preserve polarizations and filtrations. Since they have to preserve
the polarizations, all the morphisms are embeddings of Hodge structures.

\subsubsection{} The category $\catAb$ of polarizable variations of Hodge structures of abelian type over $\Delta^*$:
the objects are pairs $(\HH,\FF^\sdot)$, where $\HH$ is a local system of $\bbQ$-vector spaces
over $\Delta^*$ with unipotent monodromy and $\FF^\sdot$ is a filtration of the bundle $\HH\otimes\OO_{\Delta^*}$.
The filtration has to be of the form $0=\FF^2\HH\subset \FF^1\HH\subset \FF^0\HH=\HH\otimes\OO_{\Delta^*}$ and has to admit a polarization,
that is, an element $\omega\in \Gamma(\Delta^*,\Lambda^2\HH^*)$, such that $\FF^1\HH$ is Lagrangian (in
particular its rank is half of the rank of $\HH$) and for any non-vanishing local section $v$
of $\FF^1\HH$ we have $\ii\omega(v,\bar{v}) > 0$. The morphisms in $\catAb$ are morphisms of local
systems that preserve the filtrations. We do not fix the polarizations and do not require that the
morphisms preserve them.

\subsubsection{} \label{subsubsec:MHS-K3}
The category $\catlimK$ of polarized mixed Hodge structures of K3 type (cf. \cite[Definition 2.26]{CKS}). Its objects are
tuples $(V,q,F^\sdot,N)$ where $V$ is a $\bbQ$-vector space, $q\in S^2V^*$ a non-degenerate symmetric bilinear
form, $N\in \frso(V,q)$ is a nilpotent operator satisfying $N^3=0$, and $F^\sdot$ a filtration on $V_\bbC$
of the form $0 = F^3V\subset F^2V\subset F^1V\subset F^0V = V_\bbC$. Moreover, these structures must satisfy the following
conditions. If we denote by $W_\sdot$ the increasing filtration on $V$ defined by $N$ (cf.\ \cite[pp.\ 106]{Mo}), with the convention
that the non-zero graded components have degrees from $0$ to $4$, then $(V,F^\sdot,W_\sdot)$ is a mixed Hodge
structure; the subspace $F^2V$ is one-dimensional and for $0\neq \sigma\in F^2V$ we have $q(\sigma,\sigma)=0$,
$q(\sigma,\bar{\sigma})>0$; the Hodge structures on the primitive parts
$P_{2+i} = \ker(N^{i+1}\colon \mathrm{gr}^W_{2+i}V\to \mathrm{gr}^W_{-i}V)$ are polarized by $q(-,N^i-)$, $i=0,1,2$.
Morphisms in $\catlimK$ are morphisms of vector spaces preserving polarizations, filtrations, and commuting with
the nilpotent operators. In particular, they are morphisms of mixed Hodge structures.

\subsubsection{} The category $\catlimAb$ of polarizable mixed Hodge structures of abelian type: objects are tuples $(H,F^\sdot,N)$, where $H$ is a $\Q$-vector space, $F^\sdot$ is a filtration on $H_\C$ with $0 = F^2H_\C\subset F^1H_\C\subset F^0H_\C = H_\bbC$ and $N$ is a nilpotent operator on $H$ with $N^2=0$ which admits a 
polarization in the sense of \cite[Definition 2.26]{CKS} (analogous to the K3 type case described
above). 
If $W_\sdot$ denotes the weight filtration associated to the operator $N$ (cf.\ \cite[pp.\ 106]{Mo}), then the tuple $(H,F^\sdot,W_\sdot)$ is a mixed Hodge structure  of type $(0,0)+(1,0)+(0,1)+(1,1)$.
Morphisms in $\catlimAb$ are morphisms of $\Q$-vector spaces which preserve $F^\sdot$ and $N$.

\begin{rem} \label{rem:polarizations}
It is important for the construction in Theorem \ref{thm:limit-KS}, that the morphisms in the categories
$\catK$ and $\catlimK$ preserve polarizations, while the morphisms in $\catAb$
and $\catlimAb$ do not. This reflects the fact that Kuga-Satake abelian varieties
are polarizable, but the polarizations on them are not canonical, in particular
they are not compatible with the automorphisms of the initial polarized K3 type
Hodge structures.
\end{rem}

\subsection{The functors computing limit mixed Hodge structures}

Let us recall the definition of the functors $\LimK$ and $\LimAb$ from the diagram (\ref{diag:functors}).
For the detailed discussion of limit mixed Hodge structures, we refer to \cite[\S4]{Sch} and \cite{CKS}.

\subsubsection{} We start from the definition of $\LimK$. Consider an object $(\VV,q,\FF^\sdot)\in\catK$. 
Let $V=\VV_t$ be the fibre of $\VV$ above the fixed base point $t\in\Delta^*$. 
Then $q$ is an element of $S^2V^*$ invariant under the monodromy transformation $T=e^N$ of $\VV$, so that $N\in\frso(V,q)$ and $T\in SO(V,q)$.

\begin{defn}\label{def:perK3}
The extended period domain $\hat{\DD}_{K3}\subset \bbP(V_\bbC)$ for $(V,q)$ is the quadric defined by $q$.
The period domain for $(V,q)$ is the open subset $\DD_{K3} = \{[v]\in \hat{\DD}_{K3}\st q(v,\bar{v})>0\}$.
\end{defn}

Our terminology follows Schmid \cite{Sch}, where (extended) period domains are defined in terms
of flag varieties. We give a slightly different definition, since in our case
the period domain $\DD_{K3}$ is a Hermitian symmetric domain of noncompact type and $\hat{\DD}_{K3}$ is its compact dual.

Recall that $\tau\colon \bbH\to \Delta^*$ is the universal covering, where $\bbH$ is the upper half-plane.
The local system $\tau^*\VV$ is trivial and so $\tau^\ast \FF^2$ defines a morphism
$\Phi_{K3}\colon \bbH\to \DD_{K3}$ that satisfies the relation $\Phi_{K3}(z + 1) = T\cdot\Phi_{K3}(z)$.
Define $\Psi_{K3}(z) = e^{-zN}\cdot \Phi_{K3}(z)$, then $\Psi_{K3}(z+1) = \Psi_{K3}(z)$. 
By the nilpotent orbit theorem \cite[4.9]{Sch}, there exists a limit $[v_{lim}] = \lim\limits_{\mathrm{Im}(z)\to+\infty} \Psi_{K3}(z)\in \hat{\DD}_{K3}$.
We define the limit Hodge filtration on $V_\bbC$ by setting $F^2_{lim}V$ to be the subspace
spanned by $v_{lim}$ and $F^1_{lim}V=(F^2_{lim}V)^\perp$.
It follows from the $SL_2$-orbit theorem \cite{Sch} that $(V,q,F^\sdot_{lim},N)\in\catlimK$ and
we define $\LimK(\VV,q,\FF^\sdot) = (V,q,F^\sdot_{lim},N)$. 
A morphism in $\catK$ is an embedding of polarized variations of Hodge structures, and its image under $\LimK$ is the corresponding embedding of polarized mixed Hodge structures.

\subsubsection{} The definition of $\LimAb$ is analogous. Consider $(\HH,\FF^\sdot)\in \catAb$. Let $H=\HH_t$ be the fibre
of $\HH$ above $t\in \Delta^*$ and let $T' = e^{N'}$ be the monodromy transformation. 
Fix a $T'$-invariant element $\omega\in \Lambda^2H^*$
defining a polarization, so that $T'\in \Sp(H,\omega)$.

\begin{defn}\label{def:perAb}
The extended period domain for $(H,\omega)$ is the Grassmannian of Lagrangian
subspaces $\hat{\DD}_{Ab} = \mathrm{LGr}(H_{\bbC},\omega)$.
The period domain for $(H,\omega)$ is the open subset
$$
\DD_{Ab} = \{[H^{1,0}]\in \hat{\DD}_{Ab}\st \ii\omega(v,\bar{v})>0, \, \, \forall v\in H^{1,0}\}.
$$
\end{defn}

Analogously to the case of K3 type Hodge structures, $\hat{\DD}_{Ab}$ is the compact dual of $\DD_{Ab}$.

We have a morphism $\Phi_{Ab}\colon U\to \DD_{Ab}$ and define $\Psi_{Ab}(z) = e^{-zN}\cdot \Phi_{Ab}(z)$.
By the nilpotent orbit theorem \cite[4.9]{Sch}
there exists a limit $[H^{1,0}_{lim}] = \lim\limits_{\mathrm{Im}(z)\to+\infty} \Psi_{Ab}(z)\in \hat{\DD}_{Ab}$.
We define the limit Hodge filtration on $H_\bbC$ by setting $F^1_{lim}H = H^{1,0}_{lim}$.
The weight filtration $W_\sdot$ is determined by the operator $N$.
It follows from the $SL_2$-orbit theorem that $(H,F^\sdot_{lim},W_\sdot)\in\catlimAb$.
We note that $H^{1,0}_{lim}$ does not depend on the choice of $\omega$, since the limit
can be taken in the Grassmannian of all half-dimensional subspaces in $H_\bbC$. Hence
we can define $\LimAb(\HH,\FF^\sdot) = (H,F^\sdot_{lim},W_\sdot)$. 
To define the action of
$\LimAb$ on morphisms, let $\varphi:(\mathcal H_1,\FF^\sdot_1) \to (\mathcal H_2, \FF^\sdot_2 )$ be a morphism of VHS on $\Delta^*$.
We aim to show that the restriction of $\varphi$ to the fibre above $t\in \Delta^\ast$ defines a morphism of mixed Hodge structures.
Since the objects in $\catAb$ are semi-simple, it is
enough to consider the case where $\varphi$ is an automorphism of a simple object. 
In this case the claim is clear.


\subsection{The category $\catAb$ and families of abelian varieties over the disc}\label{subsec:catAb}
Up to a finite covering of $\Delta^*$, any variation of Hodge structures $(\HH,\FF^\sdot)\in\catAb$ defines a family of abelian varieties $\alpha'\colon\AA^*\to \Delta^*$ (unique up to isogeny). 
We would like to fill in a central fibre, producing a flat projective
family $\alpha\colon\AA\to \Delta$, smooth over $\Delta^*$. This is always possible
by Borel's theorem, as we are going to recall next.

Consider the fibre $H=\HH_t$ of the local system $\HH$ and fix a polarization $\omega\in\Lambda^2H^*$.
Consider the universal covering $\tau\colon \bbH\to \Delta^*$. The VHS $(\HH,\FF^\sdot)$
induces a period map $\tilde{p}\colon \bbH\to \DD_{Ab}$, where $\DD_{Ab}$ is the period domain
for $(H,\omega)$ defined above (Definition \ref{def:perK3}). Up to a finite covering
of $\Delta^*$, we may assume that the monodromy $T'$ is contained in a torsion-free
arithmetic subgroup $\Gamma\subset \Sp(H,\omega)$. Then we get a holomorphic map
$p\colon\Delta^*\to \DD_{Ab}/\Gamma$. 
By \cite{BB}, $\DD_{Ab}/\Gamma$ is quasi-projective,
and $\AA^*$ is defined as the pull-back of a polarized family of abelian varieties
over $\DD_{Ab}/\Gamma$. 
By \cite{Bor}, the map $p$ extends to a map $\bar{p}\colon \Delta\to X$,
where $X$ is some projective compactification of $\DD_{Ab}/\Gamma$. Then
$\AA$ can be defined as the pull-back along $\bar{p}$ of a projective family over $X$.
By the semi-stable reduction theorem \cite{KKMSD}, we may also assume that $\alpha$ is semi-stable.

Applying this procedure to $\KS(\VV,q,\FF^\sdot)$ for some $(\VV,q,\FF^\sdot)\in\catK$,
we get a degenerating family of abelian varieties over $\Delta$, which we call the Kuga--Satake
family attached to $(\VV,q,\FF^\sdot)$. Note that this family is not canonically defined,
and the central fibre $\AA_0$ is not unique. However, the invariants of $\AA_0$ that we compute
do not depend on any particular choices.

By the above construction, the family $\alpha\colon\AA\to \Delta$ is projective over
the disc, meaning that $\AA$ is a complex submanifold in $\bbP^n\times\Delta$ for
some $n$, and $\alpha$ is induced by projection to the second factor. This gives
a flat family of subvarieties in the projective space, and this family is the
pull-back of the universal family over the Hilbert scheme via a holomorphic map
$f\colon \Delta\to \mathrm{Hilb(\bbP^n)}$. 
Since the analytic Hilbert scheme is given by analytification of the algebraic Hilbert scheme (see e.g.\ \cite[Proposition 5.3]{Si}), we see that $\AA$ is defined by finitely many polynomials with complex-analytic coefficients, and we can define the abelian
variety $\AA_K$, which is the fibre of $\AA$ over the spectrum of $K=\bbC((t))$.
This can be used to apply the construction of \cite{BLR} and obtain the N\'eron model $\AA^{\Ne}\to \Spec R$ of $\AA$, where $R=\C[[t]]$
and $\Spec R$ is the formal disc.

\section{The functor $\KS$} \label{sec:ks}

In this section we define the functor $\KS$ from the diagram (\ref{diag:functors}).
Consider an object $(\VV,q,\FF^\sdot)\in \catK$. Let $V=\VV_t$ be the fibre of $\VV$ above the fixed base point $t\in \Delta^*$ and let
$T = e^N$ be the monodromy operator. Then $q\in S^2V^*$ is a $T$-invariant element
and we have $N\in\frso(V,q)$ and $T\in \SO(V,q)$. 
To define $\KS(\VV,q,\FF^\sdot)$, we need to define a local
system $\HH$ and a Hodge filtration on $\HH\otimes\OO_{\Delta^*}$.
We will denote by $\Cl(V,q)$  the Clifford algebra of $(V,q)$, i.e.\ the quotient of the tensor algebra $T^\sdot(V) $ by the ideal generated by $v\otimes v-q(v,v)$.
Similarly,  $\Cl^+(V,q)$ denotes the even Clifford algebra of $(V,q)$, i.e.\ the sub-algebra of $\Cl(V,q)$, generated by even tensors.

\subsection{The local system} The fibre $H=\HH_t$ of $\HH$ above $t$ is defined to be the
$\bbQ$-vector space $H := \Cl(V,q)$. 
To obtain $\HH$, we need to define a monodromy transformation
$T'\in \GL(H)$, and we do this by lifting $T$ to the group $\Spin(V,q)$.

Recall that the Clifford group is defined as $G = \{g\in \Cl^\times(V,q)\,|\, \alpha(g)V g^{-1} = V\}$,
where $\alpha$ is the parity involution. We have the norm homomorphism $\mathcal{N}\colon G\to \bbQ$,
$g\mapsto g\bar{g}$, where $g\mapsto \bar{g}$ is the natural anti-involution of the Clifford
algebra. 
By definition, $\Spin(V,q) = \ker(\mathcal{N})\cap \Cl^+(V,q)$.
Recall that we have the embedding 
\begin{align}\label{eq:Lambda2}
\eta'\colon \Lambda^2V\hrarr \Cl(V,q),\ \ x\wdg y\mapsto \frac{1}{4}(xy-yx) .
\end{align}
If we use the isomorphism
\begin{align} \label{eq:iso:soV}
\Lambda^2V\stackrel{\sim}\longrightarrow \frso(V,q),\ \ v\wedge w\mapsto q(w,-)v-q(v,-)w 
\end{align}
to identify $\Lambda^2V$ with $\frso(V,q)$, then $\eta'$ induces a homomorphism of Lie algebras
\begin{align} \label{eq:eta}
\eta: \frso(V,q)\longrightarrow \Cl(V,q) ,
\end{align}
which induces an isomorphism of $\frso(V,q)$ with the sub Lie algebra of $\Cl(V,q)$, spanned by the commutators of elements of $V$.

Define 
$$
N' := \eta(N) \ \ \text{and}\ \ T' := e^{N'},
$$ 
and let $\HH$ be the local system with fibre $\HH_t=H$ and monodromy $T'$.
Define the embedding 
\begin{equation}\label{eqn:ks}
\mathrm{ks}'\colon V\hrarr \emrp(H),\quad \mathrm{ks}'(v) = (w\mapsto vw).
\end{equation}

\begin{lem}\label{lem:monodromy}
The operator $T'$ is the unique unipotent lift of $T$ to $Spin(V,q)$,
and $\mathrm{ks}'$ induces an embedding of local systems $\mathrm{ks}'\colon\VV\hrarr \emrp(\HH)$.
\end{lem}
\begin{proof}
For any $g\in \frso(V,q)$ and any $x\in V\subset \Cl(V,q)$,
we have $g\cdot x = \mathrm{ad}_{\eta(g)}x = [\eta(g),x]\in V$.
For $a = \eta(g)$, $x\in V$ and a formal variable $s$,
$$
e^{sa}xe^{-sa} = \sum_{i\ge 0} \frac{1}{i!}\mathrm{ad}_a^i(x) s^i = \sum_{i\ge 0} \frac{1}{i!}(g^i\cdot x) s^i.
$$ 
When $a$ is nilpotent, $e^{sa}$ is a polynomial in $s$, and we see that $e^{sa}\in \Spin(V,q)$.
Hence the two possible lifts of $T$ to the $\Spin$-group are $T' = e^{N'}$ and $T'' = -e^{N'}$, and only $T'$ is unipotent.

To prove that $\mathrm{ks}'$ defines a map of local systems, we note that
the monodromy transformation of $\emrp(\HH)$ is the conjugation by $T'$,
and by the formula above, $T\cdot v = T'v(T')^{-1}$ for any $v\in V$.
This concludes the lemma.
\end{proof}

\begin{rem} \label{rem:weightfiltration}
Since the monodromy operator $T$ on $V$ respects the bilinear form $q$, it induces a natural operator on the Clifford algebra $\Cl(V,q)$, given by $v_1\cdots v_k\mapsto T(v_1)\cdots T(v_k)$.
However, that operator does not coincide with the monodromy operator $T'=e^{N'}$ defined above.
In fact, while $T'$ restricts to the action of $T$ on the image of $V$ inside $\End(H)$, such a compatibility statement fails for the  ``naive operator'' on $H=\Cl(V,q)$, considered above.
Similarly, while $T'$ is unipotent of index two, the above operator does not have that property.
\end{rem}

\subsection{The Hodge filtration} \label{sec:Hodgefiltration}

In this section we show (see Proposition \ref{prop:ks_morphism}) that the Kuga--Satake construction extends to the extended period domain $\hat \DD_{K3}$.
To this end, we introduce a description of the Kuga--Satake construction (see Lemma \ref{lem:KS-alternative}), which might be of independent interest. 

\subsubsection{}
Let $(V,q)$ be a Hodge structure of K3 type. 
Let $v\in V^{2,0}$ be a generator with $q(v,\bar{v})=2$.
Consider $e_1 = \mathrm{Re}(v)$, $e_2 = \mathrm{Im}(v)$ and $I_v = e_1e_2$ (product
in the Clifford algebra). We have $I_v^2 = -1$ and the left multiplication by $I_v$
defines a complex structure on the vector space $H_\bbR$. The corresponding weight
one Hodge structure on $H$ is called the Kuga--Satake Hodge structure, see \cite[Chapter 4]{Huy}.

The Kuga--Satake Hodge structures are polarized and the polarization
can be defined as follows.
Pick a pair of elements $a_1,a_2\in V$, such that $q(a_1,a_1)>0$, $q(a_2,a_2)>0$ and
$q(a_1,a_2)=0$. Let $a = a_1a_2\in H$. Define the two-form $\omega\in \Lambda^2H^*$
by $\omega(x,y) = \mathrm{Tr}(xa\bar{y})$, where we use the trace in the Clifford algebra.
It is known that either $\omega$ or $-\omega$ defines a polarization, see \cite[Chapter 4]{Huy}.

\begin{lem}\label{lem_polarization}
The two-form $\omega$ defined above is $\Spin(V,q)$-invariant.
\end{lem}
\begin{proof}
Let $g\in \Spin(V,q)$. Then, using $\bar{g}g = 1$, we get: $\omega(gx,gy) = \mathrm{Tr}(gxa\bar{y}\bar{g})=
\mathrm{Tr}(xa\bar{y}\bar{g}g) = \mathrm{Tr}(xa\bar{y}) = \omega(x,y)$.
\end{proof}

Since the monodromy of the local system constructed above is contained in $\Spin(V,q)$,
the above lemma shows that the polarization given by $\pm\omega$ is monodromy-invariant.
Hence it will define a polarization of the corresponding VHS.

Recall from Definitions \ref{def:perK3} and \ref{def:perAb}
the extended period domains $\hat{\DD}_{K3}$ and $\hat{\DD}_{Ab}$ for $(V,q)$ and $(H,\omega)$.
Note that both $\hat{\DD}_{K3}$ and $\hat{\DD}_{Ab}$ have natural $\Spin(V_\bbC,q)$-actions:
on $\hat{\DD}_{K3}$ the action comes from the canonical homomorphism
$\Spin(V_\bbC,q)\to \SO(V_\bbC,q)$, and on $\hat{\DD}_{Ab}$ the action is induced
from the structure of a left $\Cl(V_\bbC,q)$-module on $H_\bbC$. This action
is faithful and preserves the form $\omega$ and so it gives rise to an embedding $\Spin(V_\bbC,q)\hrarr \mathrm{Sp}(H_\bbC,\omega)$.

As a point in $\DD_{Ab}$, the Kuga--Satake Hodge structure is given by a subspace $H^{1,0} \subset H_\bbC$ of dimension $\frac{1}{2}d$, where $d := \dim_{\bbQ}(H)$.

\begin{lem}\label{lem:KS-alternative}
Let $(V,q)$ be a K3 type Hodge structure.  
Then the corresponding Kuga--Satake Hodge structure of weight one on $H=\Cl(V)$ is given by the half-dimensional subspace 
$H^{1,0} = V^{2,0}\cdot \Cl(V_\bbC,q)$. 
\end{lem}

\begin{proof}
Let $v\in V^{2,0}$ be a generator with $q(v,\overline v)=2$, and let $e_1 = \mathrm{Re}(v)$ and  $e_2 = \mathrm{Im}(v)$.
Then, $H^{1,0}$ is the $i$-eigenspace of the operator $I_v:H_\C\to H_\C$, given by left multiplication with $I_v = e_1e_2$.
It is thus clear that $H^{1,0}$ is a right ideal.
To see  $V^{2,0}\subset H^{1,0}$, note that $q(e_1,e_1)=q(e_2,e_2)=1$ and $q(e_1,e_2)=0$ and so we get the following identity in the Clifford algebra $H=\Cl(V,q)$:
$$
e_1e_2\cdot v=-e_2+ie_1=iv .
$$
Hence, $I_v\cdot v=iv$, which proves $V^{2,0}\subset H^{1,0}$.
To conclude the lemma, it now suffices to check $$
\dim(v\,\Cl(V_\bbC,q)) = \frac{1}{2} d.
$$
This follows from the next lemma, which concludes the proof.
\end{proof}

\begin{lem}\label{lem:ideal} 
For any $[v]\in \hat{\DD}_{K3}$, the right ideal $v\cdot \Cl(V_\bbC,q)$
has dimension $\frac{1}{2}d$.
\end{lem}
\begin{proof} 
Choose an element
$w\in V_\bbC$ with $q(w,w)=0$, $q(v,w) = 1$ and denote by $V'$ the orthogonal
complement to $W = \langle v,w\rangle$. Then $\Cl(V_\bbC,q)\simeq \Cl(W,q|_W)\otimes \Cl(V',q|_{V'})$
and $v\,\Cl(V_\bbC,q) = (v\,\Cl(W,q|_W))\otimes \Cl(V',q|_{V'})$. 
Since $\Cl(W,q|_W)=\langle 1,v,w,vw \rangle$ and $v^2=0$, we see that
$v\,\Cl(W,q|_W) = \langle v,vw\rangle$ and the claim follows.
\end{proof}

Lemmas \ref{lem:KS-alternative} and \ref{lem:ideal} show that the Kuga--Satake correspondence extends to $\hat{\DD}_{K3}$,
an observation that we will need later.

\begin{prop}\label{prop:ks_morphism}
There exists a $\Spin(V_\bbC,q)$-equivariant morphism $\kappa\colon \hat{\DD}_{K3}\to\hat{\DD}_{Ab}$,
whose restriction to $\DD_{K3}$ maps a K3 type Hodge structure to the corresponding Kuga--Satake Hodge structure.
\end{prop}
\begin{proof}
Over $\hat{\DD}_{K3}$ we have the universal subbundle $\OO_{\hat{\DD}_{K3}}(-1)\hrarr V\otimes \OO_{\hat{\DD}_{K3}}$.
Using the embedding $V\hrarr \Cl(V,q) = H$ we get a subbundle $\OO_{\hat{\DD}_{K3}}(-1)\hrarr H\otimes \OO_{\hat{\DD}_{K3}}$.
Lemma \ref{lem:ideal} shows that $\OO_{\hat{\DD}_{K3}}(-1)\cdot(H\otimes \OO_{\hat{\DD}_{K3}})\subset H\otimes \OO_{\hat{\DD}_{K3}}$
is a subbundle of rank $\frac{1}{2}d$. This defines the morphism $\kappa$. To check
that it is $\Spin(V_\bbC,q)$-equivariant, let $g\in \Spin(V_\bbC,q)$. We have
$g\cdot [v] = [gvg^{-1}]$ where we use multiplication in the Clifford algebra.
Then $\kappa(g\cdot [v]) = [gvg^{-1}\,\Cl(V_\bbC,q)] = [gv\,\Cl(V_\bbC,q)] = g\cdot[v\,\Cl(V_\bbC,q)] = g\cdot \kappa([v])$.
\end{proof}

\subsubsection{} We go back to the construction of the Hodge filtration on the local system $\HH$.
Recall that we have started from an object $(\VV,q,\FF^\sdot)\in\catK$ and $V=\VV_t$ is the fibre of the
local system $\VV$ at the base point $t\in \Delta^\ast$.
We consider the universal covering $\tau\colon \bbH\to \Delta^*$ and the subbundle $\tau^*\FF^2\VV$ of
the trivial bundle $V\otimes\OO_U$. This defines a period map $p_{K3}\colon \bbH\to \DD_{K3}$.
Let $p_{Ab} = \kappa\circ p_{K3}$ and let $E$ be the pull-back of the universal vector bundle
over $\hat{\DD}_{Ab}$. Since $\kappa$ is $\Spin(V,q)$-equivariant by Proposition \ref{prop:ks_morphism}
and the monodromy operator $T'$ lies in $\Spin(V,q)$, the bundle $E$ descends to a subbundle $\tilde{\FF}^1\subset \HH\otimes\OO_{\Delta^*}$.
This defines a Hodge filtration $\tilde{\FF}^\sdot$ on $\HH$. Note that by construction we get
a polarizable variation of Hodge structures. We define $\KS(\VV,q,\FF^\sdot) = (\HH,\tilde{\FF}^\sdot)$.
The action of $\KS$ on morphisms is clear from the construction (note that morphisms in $\catK$ preserve
polarizations, so they induce embeddings of the corresponding Clifford algebras).

\begin{lem}\label{lem:ks}
The formula (\ref{eqn:ks}) defines an embedding of variations of Hodge structures
$$
\mathrm{ks}'\colon \VV(1)\hrarr\emrp(\HH). 
$$
\end{lem}
\begin{proof}
We use Lemma \ref{lem:monodromy}. It remains to check that $\mathrm{ks}'$ respects the Hodge
filtration, but this is true pointwise (see \cite[Chapter 4]{Huy}).
\end{proof}

\section{Weight filtration of the Kuga--Satake MHS} \label{sec:weight}

Let $(\VV,q,\FF^\sdot)\in\catK$ and $(\HH,\tilde{\FF}^\sdot) = \KS(\VV,q,\FF^\sdot)$.
Let $V=\VV_t$ be the fibre of $\VV$ and $H = \Cl(V,q)$ be the fibre of $\HH$ at $t\in \Delta^\ast$ (recall the
definition of the functor $\KS$ in Section \ref{sec:ks}). Recall from Section
\ref{sec:preliminaries} that the weight filtration
on $\overline{H} = \LimAb(\HH,\tilde{\FF}^\sdot)$ is determined by the logarithm $N'$
of the monodromy operator $T'$. In this section we compute the dimensions
of its components.

Recall that the monodromy operator of $\VV$ is $T=e^N$, where $N\in\frso(V,q)$
satisfies $N^3 = 0$. The limit weight filtration on $V$ is of the form
$0=W_{-1}V\subset W_0V \subset W_1V\subset W_2V \subset W_3V\subset W_4V=V$,
and the components can be described as follows:
$$
W_0V = \im(N^2), \ \ W_3V=\ker(N^2), \ \ W_1V=N(W_3V), \ \ W_2V = N^{-1}(W_0V) .
$$
We recall that $(\VV,q,\FF^\sdot)$ is of type I if $N=0$, of type II
if $N\neq 0$, $N^2=0$ and of type III if $N^2\neq 0$. The case of type I
is trivial and we do not consider it.

The limit weight filtration on $H$ is of the form
$0=W_{-1}H\subset W_0H = \im(N')\subset W_1H = \ker(N')\subset W_2H= H$.
We denote by $r$ the rank of $\VV$ and by $d = 2^r$ the rank of $\HH$.
 
\begin{prop} \label{prop:weight-filtration}
Suppose that $N\neq 0$.
Then the image of the operator $N:V\to V$ is two-dimensional. 
Under the homomorphism $\eta:\frso(V,q)\to \Cl(V,q)$ from (\ref{eq:eta}),
the image $N'=\eta(N)$ 
is proportional to the bivector corresponding to the image of $N:V\to V$. 
Moreover,\\
1) in the type II case: $\dim(W_0H) = \frac{1}{4}d$ and $\dim(W_1H) = \frac{3}{4}d$;\\
2) in the type III case: $W_0H = W_1H$ and $\dim(W_0H) = \frac{1}{2}d$.
\end{prop}
\begin{proof}
In the computations below, we will use the following fact: if $v\in V$ is an
isotropic element, then we can find a hyperbolic plane $\langle x,y\rangle\subset V$
with $q(x,x) = 2$, $q(y,y) = -2$, $q(x,y) = 0$ and $2v = x+y$.
To prove this fact, note that $q$ is non-degenerate and so we can choose an element $z\in V$ with $q(v,z)=1$.
The element $w:=z-\frac{1}{2}q(z,z)v$ has then the property that $q(w,w)=0$ and $q(v,w)=1$.
Putting $x:=v+w$ and $y:=v-w$, we obtain a hyperbolic plane $\langle x,y\rangle=\langle v,w\rangle \subset V$ with $q(x,x)=2$, $q(y,y) = -2$, $q(x,y) = 0$ and $2v = x+y$, as claimed.
This proves the above fact.

{\bf Type II case.} In this case $N\neq 0$, $N^2 = 0$, hence $W_0V = 0$, $W_1V=\im(N)$, $W_2V=\ker(N)$ and $W_3V = V$.
Since $\FF^2\VV$ is of rank one, the Hodge numbers of the limit mixed Hodge
structure $\LimK(\VV,q,\FF^\sdot)$ are $h^{0,0} = h^{2,2} = 0$, $h^{1,0} = h^{0,1} = h^{2,1} = h^{1,2} = 1$,
$h^{1,1} = r - 4$. 
It follows that $\dim W_1V = 2$ and so the image of $N$ is two-dimensional.
For any $x,y\in V$, we have $q(Nx,Ny) = -q(x,N^2y)= 0$, so $W_1V$ is an isotropic subspace.

Let now $e_1\in V$ be a non-isotropic element. 
Since $W_1V=\im(N)$ is a non-trivial isotropic subspace, we can by the above fact assume that $e_1$ is contained in a hyperbolic plane and $q(e_1,e_1) = 2$.
Then $Ne_1$ is isotropic and $q(Ne_1,e_1) = 0$.
By the above fact, applied to the isotropic vector $Ne_1\in e_1^{\perp}$, we can find a hyperbolic plane $U\subset e_1^\perp$
with a basis $U=\langle e_2,e_3\rangle$, $q(e_2,e_2)=2$, $q(e_2,e_3)=0$, $q(e_3,e_3)=-2$,
such that $2Ne_1=e_2+e_3$. Then $q(Ne_2, e_2) = 0$, $q(Ne_2,e_3) = q(Ne_2, 2Ne_1-e_2) = 0$,
so that $Ne_2 = -Ne_3$ is orthogonal to $U$. We also have $q(2Ne_3,e_1) = -q(e_3,e_2+e_3) = 2$.
Let $e_4 = 2Ne_3-e_1$. Then $q(e_1,e_4) = 0$, $q(e_4,e_4) = -2$, and $U'=\langle e_1,e_4\rangle$
is a hyperbolic plane orthogonal to $U$. 

Since $\im(N)\subset U\oplus U'$, the orthogonal complement to $U\oplus U'$ is
contained in the kernel of $N$. 
The action of $N$ on $U\oplus U'$ is:
$$
Ne_1 = -Ne_4 = \frac{1}{2}(e_2+e_3);\quad Ne_3 = -Ne_2 = \frac{1}{2}(e_1+e_4).
$$

Under the isomorphism $\frso(V,q)\simeq \Lambda^2V$ from (\ref{eq:iso:soV}), the operator $N$
is represented by $\frac{1}{4}(e_2+e_3)\wdg (e_1+e_4)$; indeed, both sides vanish on the orthogonal complement of $U\oplus U'$ and they agree on the basis $e_1$, $e_2$, $e_3$ and $e_4$ of $U\oplus U'$.
Since $U$ and $U'$ are orthogonal to each other, we get  via (\ref{eq:Lambda2}):
$$
N'=\eta (N)=\frac{1}{8}(e_2+e_3)(e_1+e_4) .
$$
This is the bivector corresponding to $\im(N)$.

It remains to compute the dimension of $W_0H=\im(N')$ and $W_1H=\ker (N')$.
Let $f_1 = \frac{1}{2}(e_1+e_4)$ and $f_2 = \frac{1}{2}(e_2+e_3)$. 
Then, $N'=\eta(N) = \frac{1}{2}f_2f_1\in \Cl(V,q)$.
Let $V'$ be the orthogonal complement of $U\oplus U'$. 
Then $\Cl(V,q) \simeq \Cl(U\oplus U',q|_{U\oplus U'})\otimes \Cl(V',q|_{V'})$.
We have $N' \in \Cl(U\oplus U',q|_{U\oplus U'})$, so it is enough to find the image and the kernel of $N'$ acting
by left multiplication on $\Cl(U\oplus U',q|_{U\oplus U'})\simeq \mathrm{Mat}_4(\bbQ)$. 
Via this last isomorphism, $N'$ corresponds to a $4\times 4$ matrix
with the only non-zero element in the upper right corner. 
The kernel of this matrix is $3$-dimensional and the image is $1$-dimensional.

{\bf Type III case.} In this case $N^2 \neq 0$, $N^3 = 0$. 
The Hodge numbers of the limit mixed Hodge
structure $\LimK(\VV,q,\FF^\sdot)$ are $h^{0,0} = h^{2,2} = 1$, $h^{1,0} = h^{0,1} = h^{2,1} = h^{1,2} = 0$,
$h^{1,1} = r - 2$. It follows that $W_0V = W_1V$, $W_2V=W_3V$ and $\dim(W_0V) = \dim(V/W_3V) = 1$.

Consider the isotropic subspace $V' = \im(N)$, which contains $W_0V=\im(N^2)$. 
Since $W_1V=N(W_3V)$ is one-dimensional and $W_3V=\ker(N^2)$ is of codimension one in $V$, we conclude that $\dim V' = 2$. 
Since $q$ is non-degenerate, we can find $x,y\in V$ such that $q(Nx,Ny) = -q(x,N^2y) \neq 0$.
It follows that $q|_{V'}$ is non-zero. Pick an element $v\in V'$ with $q(v,v)\neq 0$.
Then, $v=Nv'$ for some $v'\in V$ and $q(v,N^2x) = 0$ for any $x\in V$.
Hence, $W_0V\subset v^\perp$ and we get an orthogonal decomposition 
$$
V' = W_0V\oplus \langle v\rangle .
$$
By the above fact, applied to $W_0V\subset v^\perp$, 
we can find a hyperbolic plane $U\subset v^\perp$, containing $W_0V$, together with a basis $U=\langle e_1,e_2\rangle$, such that $q(e_1,e_1)=2$, $q(e_1,e_2)=0$,
$q(e_2,e_2)=-2$ and $W_0V=\langle e_1+e_2\rangle$.

Let $V'' = U\oplus \langle v\rangle$, then $\im(N)=\langle v, e_1+e_2 \rangle\subset V''$ and so $N$ preserves $V''$. 
Since $e_1+e_2\in\ker(N)$,
 $Nv \in W_0V$ is nonzero. 
Hence, there is some $\alpha\in \Q$ such that $e_3 = \alpha v$ satisfies $Ne_3 = e_1+e_2$.
Using again $e_1+e_2\in\ker(N)$, we get $Ne_1=-Ne_2$. 
From the equalities $q(Ne_1,e_1)=0$ and $q(Ne_2,e_2) = 0$, we conclude that
$Ne_1=-Ne_2\in U^\perp$ and so $Ne_1=-Ne_2=ae_3$ for some $a\in \Q$.
To find the coefficient $a$, note that $aq(e_3,e_3) = q(Ne_1,e_3) = -q(e_1, Ne_3) = -2$, so that $a = -2/q(e_3,e_3)$.
We can summarize:
$$
Ne_1 = \frac{-2e_3}{q(e_3,e_3)};\quad Ne_2 = \frac{2e_3}{q(e_3,e_3)};\quad Ne_3 = e_1+e_2.
$$

Since $U\subset v^\perp$ and $q(v,v)\neq 0$, we have an orthogonal direct sum decomposition $V=V''\oplus (V'')^\perp$. 
For any $x\in (V'')^\perp$ and $y\in V$, we have $q(Nx,y) = -q(x,Ny) = 0$, because $Ny\in V''$.
Hence, $(V'')^\perp\subset \ker(N)$, and so 
the image of $N$ under the isomorphism $\frso(V,q)\simeq \Lambda^2V$ from (\ref{eq:iso:soV}) is contained in $\Lambda^2(V'')$. 
Checking the evaluation on the basis $e_1$, $e_2$, $e_3$ of $V''$, one sees that 
$N$ is represented by
$1/q(e_3,e_3)(e_1+e_2)\wdg e_3\in \Lambda^2V$.
Using (\ref{eq:Lambda2}), we therefore get
$$
N'=\eta(N)=\frac{1}{2q(e_3,e_3)}(e_1+e_2)e_3 ,
$$
because $e_1+e_2$ and $e_3$ anti-commute in $H$.
This shows that $N'$ is proportional to the bi-vector which corresponds to the image of $N\colon V\to V$.

It remains to compute the dimension of $W_0H=\im(N')$ and $W_1H=\ker (N')$.
We have $N' = \frac{1}{2q(e_3,e_3)}(e_1+e_2)e_3\in \Cl(V,q)$. Since the element $e_3$ is invertible
in the Clifford algebra, we have $\im(N') = \ker{N'} = (e_1+e_2)\Cl(V,q)$ -- the right ideal
generated by the isotropic vector $e_1+e_2$. The dimension of this ideal is $\frac{1}{2}d$.
This finishes the proof of the proposition.
\end{proof}

\section{Proof of the main results}

\begin{proof}[Proof of Theorem \ref{thm:limit-KS}]
1) We define the functor $\limKS$. 
Let $\overline{V} = (V,q,F^\sdot,N) = \LimK(\VV,q,\FF^\sdot)$.
Let $H = \Cl(V,q)$, $N' = \eta(N)$ (cf. Section \ref{sec:ks}). 
The subspace $F^2V$ gives
a point in the extended period domain $\hat{\DD}_{K3}$ (see Definition \ref{def:perK3}).
Let $[\tilde{F}^1H] = \kappa([F^2V])$, where $\kappa$ is defined in Proposition \ref{prop:ks_morphism}.
We need to check that $(H,\tilde{F}^\sdot,N')$ is a polarizable mixed Hodge structure.

Let $(\HH,\tilde{\FF}^\sdot) = \KS(\VV,q,\FF^\sdot)$. As before, $\tau\colon \bbH \to \Delta^*$ is
the universal covering. Recall the construction of limit mixed Hodge structures from
Section \ref{sec:preliminaries}. 
We have $\Phi_{K3}\colon \bbH\to \hat{\DD}_{K3}$ and
$\Psi_{K3}(z) = e^{zN}\cdot \Phi_{K3}(z)$. Then $\Phi_{Ab} = \kappa\circ\Phi_{K3}$ and
since $\kappa$ is $\Spin(V_\bbC,q)$-equivariant by Proposition \ref{prop:ks_morphism},
we have $\Psi_{Ab} = \kappa\circ\Psi_{K3}$. This implies that
$$\lim\limits_{\mathrm{Im}(z)\to+\infty} \Psi_{Ab}(z) = \kappa \left(\lim\limits_{\mathrm{Im}(z)\to+\infty} \Psi_{K3}(z)\right),$$
and thus $(H,\tilde{F}^\sdot,N') = \LimAb\circ\KS(\VV,q,\FF^\sdot)$ is a polarizable mixed
Hodge structure. We define $\limKS(\overline{V}) = (H,\tilde{F}^\sdot,N')$ and automatically
get $\LimAb\circ\KS = \limKS\circ \LimK$.
\\

2) Given $\overline{V} = (V,q,\FF^\sdot,N)$ and $\overline{H} = (H,\tilde{F}^\sdot,N') = \limKS(\overline{V})$,
we define 
$$
\mathrm{ks}\colon \overline{V}(1)\hrarr \emrp(\overline{H}),\ \ v\mapsto f_v,
$$ 
where
$f_v(w) = vw$.

We need to check that the embedding $\mathrm{ks}$ is compatible with the weight and Hodge filtrations.
For the weight filtration, it is enough to check compatibility with the action of the monodromy operators $T=e^N$ and $T'=e^{N'}$.
This was done in Lemma \ref{lem:monodromy}.


For the Hodge filtration, one can use Lemma \ref{lem:ks} and a straightforward limit argument.
For completeness, we check compatibility directly.
The Hodge filtrations on $H^*$ and $\emrp(H)$ have the following components: 
\begin{eqnarray}
&&F^{-1}H^* = H^*, \quad F^0H^* = \{\varphi\in H^*\st \varphi|_{F^1H} = 0\},\nonumber\\
&&F^{-1}\emrp(H) = \emrp(H),\quad F^0\emrp(H) = F^1H\otimes H^* + H\otimes F^0H^*,\quad F^{1}\emrp(H) = F^1H\otimes F^0H^*.\nonumber
\end{eqnarray}

Fix $v_0\in V_\bbC$ with $q(v_0) = 0$. 
We have $F^1V(1) = \langle v_0\rangle$ and $F^1H = v_0\,\Cl(V_\bbC,q)$
by Lemma \ref{lem:ideal}.
The image of $f_{v_0}$ is contained in $F^1H$, and since $v_0$ is isotropic, $f_{v_0}$ clearly annihilates $F^1H$.
So we have $f_{v_0} \in F^1\emrp(H)$.

Next consider $v\in F^0V(1)$, so that $q(v,v_0) = 0$. Choose a subspace $W\subset H$ complementary
to $F^1H$ and let $\mathrm{pr}_W$, $\mathrm{pr}_{F^1H}$ be the projectors. Then,
$f_v = \mathrm{pr}_{F^1H}\circ f_v + \mathrm{pr}_W\circ f_v$ and $\mathrm{pr}_{F^1H}\circ f_v\in F^1H\otimes H^*$.
Also, $\mathrm{pr}_W\circ f_v(v_0x) = \mathrm{pr}_W(vv_0x) = -\mathrm{pr}_W(v_0vx) = 0$,
so $\mathrm{pr}_W\circ f_v\in H\otimes F^0H^*$ and $f_v \in F^0\emrp(H)$. This completes the proof.
\end{proof}

\begin{proof}[Proof of Theorem \ref{thm:type}]
Consider the fixed base point $t\in \Delta^\ast$.
The Clemens--Schmid sequence \cite{Mo} for weight one yields an exact sequence
$$
0\longrightarrow H^1(\AA_0,\Q)\longrightarrow H^1_{lim}(\AA_t,\Q)\stackrel{N}\longrightarrow H^1_{lim}(\AA_t,\Q).
$$
Since $\ker(N)=W_1H^1_{lim}(\AA_t,\Q)$, we conclude 
$$
H^1(\AA_0,\Q)\cong W_1H^1_{lim}(\AA_t,\Q) .
$$
The Hodge numbers of the mixed Hodge structure $H^1_{lim}(A_t,\Q)$ are computed via Proposition \ref{prop:weight-filtration}, which yields the claimed result.
This concludes Theorem \ref{thm:type}.
\end{proof}

\begin{proof}[Proof of Corollary \ref{cor:dualcomplex}]
By the definition of the weight filtration on the simple normal crossing variety $\AA_0$, we have $H^k(\Sigma,\Q)\cong W_0H^k(\AA_0,\Q)$.
The Clemens--Schmid sequence \cite{Mo} further implies $W_0 H^k_{lim}(\AA_t,\Q)\cong W_0H^k(\AA_0,\Q)$ and so we conclude
$$
H^k(\Sigma,\Q) \cong  W_0H^k_{lim}(\AA_t,\Q) .
$$
Moreover, the isomorphism $H^k(\AA_t,\Q)\cong \Lambda^k H^1(\AA_t,\Q)$ induces a canonical isomorphism of limit mixed Hodge structures
$$
H^k_{lim}(\AA_t,\Q)\cong \Lambda^k H^1_{lim}(\AA_t,\Q),
$$
see for instance \cite[Proposition 6.1]{HN2}.
Putting both identities together, we obtain
$$
H^k(\Sigma,\Q) \cong\Lambda^k  (W_0 H^1_{lim}(\AA_t,\Q))
$$
By \cite[Lemma 6.16]{KLSV}, the above isomorphisms are compatible with cup product. 
Hence, 
$$
H^\ast(\Sigma,\Q)\cong \Lambda^\ast (W_0H^1_{lim}(\AA_t,\Q))
$$ 
is isomorphic to the rational cohomology algebra of a real torus of real dimension $2\dim_\C(W_0H^1_{lim}(\AA_t,\Q))$.
The corollary follows therefore from Theorem \ref{thm:type}.
\end{proof}

\begin{proof}[Proof of Corollary \ref{cor:Neron}]
Since $\mathcal A_{0}^{\Ne}$ is a group scheme over $\C$, all components are isomorphic to the component $A:=\mathcal A_{00}^{\Ne}$ which contains the identity element.
Since $\mathcal A$ is a semi-stable model, the monodromy operator is unipotent of index two and so \cite[IX.3.5]{SGA} implies that the generic fibre $\mathcal A_K$ of the N\'eron model has semi-abelian reduction; that is, $A$ is a semi-abelian variety over $\C$.
The Chevalley decomposition of $A$ thus reads as follows
$$
0\longrightarrow T\longrightarrow A \longrightarrow B\longrightarrow 0 ,
$$
where $T\cong (\C^\ast)^w$ is an algebraic torus and $B$ is an abelian variety, cf.\ \cite[(3.1)]{HN2}.
In \cite[Theorem 6.2]{HN2}, Halle and Nicaise describe the dimensions of $T$ and $B$ as well as the rational Hodge structure of $B$ in terms of the limit mixed Hodge structure $H^1_{lim}(\AA_t,\Q)$.
The asserted description of the special fibre $\mathcal A_{0}^{\Ne}$ of the N\'eron model follows therefore from the main results of this paper, see Theorems \ref{thm:limit-KS} and \ref{thm:type} above. 

As we will explain next, the expression for the motivic zeta-function can be deduced from \cite[Proposition 8.3]{HN1}; we are grateful to Johannes Nicaise for pointing this out.
In the notation used in \textit{loc.\ cit.}, that result reads as follows:
$$
Z_{\mathcal A_K}(T)=\sum_{d\in \N} \phi_{\mathcal A_K}(d)\cdot (\mathbb L-1)^{t_{\mathcal A_K}(d)}\cdot \mathbb L^{u_{\mathcal A_K}(d)+ord_{\mathcal A_K}(d)}\cdot [B_{\mathcal A_K}(d)]\cdot T^d.
$$  
Since $\mathcal A_K$ has semi-abelian reduction, the unipotent rank $ u_{\mathcal A_K}(d)$ vanishes for all $d$, cf.\ \cite[Section 2]{HN2}.
For the same reason, \cite[Proposition 6.2]{HN1} implies that the toric rank $t_{\mathcal A_K}(d)$ and the abelian quotient $B_{\mathcal A_K}(d)$ do not depend on $d$. 
Vanishing of the
order function $ord_{\mathcal A_K}(d)$ can be deduced e.g.\ from \cite[Proposition 7.5 and Corollary 4.20]{HN1},
and for the expression $\phi_{\mathcal A_K}(d) = \phi_{\mathcal A_K}(1)d^w$, see the proof of Theorem 8.6 in \cite{HN1} and
references therein. 
Finally, $\phi_{\mathcal A_K}(1)$ is the number of components of $\AA^{\Ne}_0$ by \cite[Definition 3.6]{HN1}.
Altogether, this establishes the formula claimed in Corollary \ref{cor:Neron}.
\end{proof}

\begin{proof}[Proof of Corollary \ref{cor:supp}]
Corollary \ref{cor:supp} is an immediate consequence of the relation between the N\'eron model $\mathcal A^{\Ne}$ and the semi-stable model $\mathcal A$, given by Jordan and Morrison in \cite[Theorem 1.4]{JM}.
\end{proof}

\section*{Acknowledgement}
We are grateful to the referee for useful remarks and to Daniel Huybrechts, Radu Laza, Johannes Nicaise and Arne Smeets for discussions.
Special thanks to Johannes Nicaise for pointing out the relevance of \cite{HN1} and \cite{HN2}.

\end{document}